\documentclass[12pt,draftcls,onecolumn]{IEEEtran}

\IEEEoverridecommandlockouts                              
\usepackage{amsmath,color}
\usepackage{amssymb}
\usepackage{euscript}
\usepackage{amscd}
\usepackage{cite}
\usepackage{tikz}
\usetikzlibrary{snakes,arrows,shapes}

\newtheorem{thm}{Theorem}

\newtheorem{lemma}[thm]{Lemma}
\newtheorem{prop}[thm]{Proposition}

\newcommand{\script}[1]{\EuScript{#1}}
\newcommand{\Tr}{\operatorname{\!{^{\mbox{\scriptsize\sf T}}}}\!}
\newcommand{\Circ}{\mathop{\rm Circ}}

\newcommand{\trace} {\mbox{\rm tr}}

\newcommand{\E}{{\mathbb E}\,}

\newcommand{\Zbb}{\mathbb Z}
\newcommand{\Tbb}{\mathbb T}

\newcommand{\yb}{\mathbf  y}

\newcommand{\eb}{\mathbf  e}

\newcommand{\gb}{\mathbf  g}

\newcommand{\db}{\mathbf  d}

 \newcommand{\cb}{\mathbf c}
\newcommand{\qb}{\mathbf q}  

\newcommand{\Ab}{\mathbf A}
\newcommand{\Bb}{\mathbf B}
\newcommand{\Cb}{\mathbf C}

\newcommand{\Eb}{\mathbf E}
\newcommand{\Fb}{\mathbf F}

\newcommand{\Hb}{\mathbf H}
\newcommand{\Ib}{\mathbf I}
\newcommand{\Jb}{\mathbf J}

\newcommand{\Mb}{\mathbf M}

\newcommand{\Pb}{\mathbf P}
\newcommand{\Qb}{\mathbf Q}

\newcommand{\Sb}{\mathbf S}
\newcommand{\Tb}{\mathbf T}

\newcommand{\Sigmab}{\boldsymbol{\Sigma}}

\begin{document}
\title{The Circulant Rational Covariance Extension Problem: The Complete Solution\thanks{This research was supported by grants from
VR, ACCESS and the Italian Ministry for Education and Research (MIUR).} }
\author{Anders Lindquist\thanks{A. Lindquist is with the Department of  Automation, Shanghai Jiao Tong University, Shanghai, China, and the Center for Industrial and Applied Mathematics (CIAM) and the ACCESS Linnaeus Center, Royal Institute of Technology, Stockholm, Sweden, {\tt alq@kth.se}}
and Giorgio Picci
\thanks{Giorgio Picci   is with the
Department of Information Engineering, University of Padova, via
Gradenigo 6/B, 35131 Padova, Italy; e-mail:  
{\tt picci@dei.unipd.it}}
}
\markboth{}
{Lindquist and Picci: Circulant Covariance Extension}
\maketitle

\begin{abstract}
The rational covariance extension problem to determine a rational spectral density given a finite number of covariance lags can be seen as a matrix completion problem  to construct an infinite-dimensional positive-definite Toeplitz matrix the north-west corner of which is given. The circulant rational covariance extension problem considered in this paper is a modification of this problem to partial stochastic realization of reciprocal and periodic stationary process, which are better represented on the discrete unit circle $\mathbb{Z}_{2N}$ rather than on the discrete real line $\mathbb{Z}$. The corresponding matrix completion problem then amounts to completing a finite-dimensional Toeplitz matrix that is  circulant. Another important motivation for this problem is that it provides a natural approximation, involving only computations based on the fast Fourier transform,  for the ordinary rational covariance extension problem, potentially leading to an efficient numerical procedure for the latter. The circulant rational covariance extension problem is an inverse problem with infinitely many solutions in general, each corresponding to a bilateral ARMA representation of the underlying periodic (reciprocal) process. In this paper we present a complete smooth parameterization of all solutions and convex optimization procedures for determining them. A procedure to determine which solution that best matches additional data in the form of logarithmic moments is also presented. 
 \end{abstract}
 
 \newpage
 
 \section{Introduction}
 
 \IEEEPARstart{T}{he rational} covariance extension problem or the {\em partial stochastic realization problem\/} has been studied in various degrees of detail in a long series of  papers \cite{Kalman,Gthesis,Georgiou1,BLGM1,BGuL,SIGEST,Byrnes-L-97,BEL1,BEL2,PEthesis,Pavon-F-12}. In a formulation suitable for this paper it can be stated as follows. Given a sequence $(c_0,c_1,\dots,c_n)$ of numbers, with $c_0$ real and the rest possibly complex, such that the Toeplitz matrix
 \begin{equation}
\label{Toeplitz}
\Tb_n=\begin{bmatrix} c_0&c_1&c_2&\cdots&c_n\\
				\bar{c}_1&c_0&c_1&\cdots& c_{n-1}\\
				\bar{c}_2&\bar{c}_1&c_0&\cdots&c_{n-2}\\
				\vdots&\vdots&\vdots&\ddots&\vdots\\
				\bar{c}_n&\bar{c}_{n-1}&\bar{c}_{n-2}&\cdots&c_0
 		\end{bmatrix}
\end{equation}
 is positive definite, find an infinite extension $c_{n+1},c_{n+3},c_{n+3}, \dots$ such that, with $c_{-k}=\bar{c}_k$, $k=1,2,\dots$, the series expansion 
 \begin{equation}
\label{Phi}
\Phi(e^{i\theta})=\sum_{k=-\infty}^\infty c_k e^{-ik\theta}
\end{equation}
converges to a positive spectral density for all $\theta\in [-\pi,\pi]$ which takes the rational form 
\begin{equation}
\label{Phi}
\Phi(z)=\frac{P(z)}{Q(z)},
\end{equation}
where $P$ and $Q$ are symmetric trigonometric polynomial of the form
\begin{equation}
\label{P(z)}
P(e^{i\theta})= \sum_{k=-n}^n p_k e^{-ik\theta}, \quad p_{-k}=\bar{p}_k,
\end{equation}
 of degree $n$ in the case of $Q$ or at most $n$ in the case of $P$. In \cite{Gthesis,Georgiou1} it was shown that there exists a $Q$ for each assignment of $P$ and in \cite{BLGM1} it was finally proved that this assignment is unique and smooth, yielding a complete parameterization suitable for tuning. Consequently, the rational covariance extension problem reduces to a trigonometric moment problem, where, for each $P$, the remaining problem is to determine a unique $Q$ such that 
\begin{equation}
\label{moments}
 \int_{-\pi}^\pi e^{ik\theta}\frac{P(e^{i\theta})}{Q(e^{i\theta})}\frac{d\theta}{2\pi}=c_k, \quad k=0,1,2,\dots, n.
\end{equation}
In \cite{BGuL,SIGEST} a convex optimization procedure to determine these $Q$ was introduced, a result that has  then been generalized in several directions \cite{BEL1,BEL2,PEthesis,BLmoments,BLimportantmoments,BLkrein,Georgiou3,Georgiou-L-03}.
 
The rational covariance extension problem can be seen as a {\em matrix completion problem} to construct an infinite-dimensional positive-definite Toeplitz matrix $\Tb_\infty$ with $\Tb_n$  in its north-west corner, which moreover satisfies the rationality constraint \eqref{Phi}. There is an large literature on band extension of positive-definite Toeplitz matrices, starting with the work of Dym and Gohberg \cite{Dym-G-81} and surveyed in the books \cite{Gohberg-G-K-94,rodman2002abstract}, dealing with the maximum-entropy solution corresponding to  $P\equiv 1$ . 

The {\em circulant rational covariance extension problem\/}, considered in this paper, is a modification of this problem to partial stochastic realization of periodic stationary process, which, as we shall explain in detail below, are better represented on the discrete unit circle $\mathbb{Z}_{2N}$ (the integers modulo $2N$) than on the the discrete real line $\mathbb{Z}$. The corresponding matrix completion problem then amounts to completing a finite-dimensional Toeplitz matrix that is {\em circulant}.   
An important motivation for this problem is that its solution is a natural approximation of the solution to the ordinary rational covariance extension problem that turns out to involve only computations based on the fast Fourier transform and seems to lead to an efficient numerical procedure.

Circulant matrices are Toeplitz matrices with a special circulant structure
\begin{equation}
\label{ }
\Circ\{ m_0,m_1,m_2,\dots, m_\nu\} =
\begin{bmatrix}m_0&m_\nu&m_{\nu-1}&\cdots&m_1\\
			m_1&m_0&m_\nu&\cdots&m_2\\
			m_2&m_1&m_0&\cdots&m_3\\
			\vdots&\vdots&\vdots&\ddots&\vdots\\
			m_\nu&m_{\nu-1}&m_{\nu-2}&\cdots&m_0
\end{bmatrix},
\end{equation}
where the columns (or, equivalently, rows) are shifted cyclically, and where  $m_0,m_1,\dots,m_\nu$ here are taken to be  complex numbers.  In the circulant rational covariance extension problem we consider {\em Hermitian\/} circulant  matrices 
\begin{equation}
\label{M_C}
\Mb:=\Circ\{ m_0,\bar{m}_1,\bar{m}_2,\dots, \bar{m}_N,m_{N-1},\dots,m_2,m_1\}.
\end{equation}

Hermitian circulant matrices appear naturally in the context of  periodic stationary  stochastic processes. To see this consider a zero-mean  stationary process $\{y(t)\}$,  in general complex-valued, defined on a finite interval $[-N+1,\,N]$ of the integer line $\Zbb$ and extended to all of $\Zbb$ as a periodic stationary process with  period $2N$; i.e., such that $y(t +2kN) =y(t)$ almost surely.  Processes of this kind can naturally be  defined on the group $\Zbb_{2N}$ of the integers with arithmetics modulo $2N$, and in this setting stationarity can be seen as propagation in time of random variables under the action of a (finite) unitary group.  We shall write the string $\{y(-N+1),\ldots, y(0),\ldots,y(N)\}$ as a $2N$-dimensional column vector $\yb$ and only consider stationary process of {\em full rank}, whose covariance matrix
\begin{equation}
\label{Sigma}
\Sigmab : =  \E\{ \yb\yb^*\}
\end{equation}
 is positive definite, where $^*$ denotes transpose conjugate.    Let $\hat{\E}\{y(t)\mid  y(s),\, s \neq t\}$ be the wide sense conditional mean of $y(t)$ given all $\{y(s),\, s \neq t\}$. Then the error process
\begin{equation}  \label{finconjn}
        d(t)     := y(t)-\hat{\E} \{ y(t)\mid y(s),\, s \neq t\}
  \end{equation}
is orthogonal to all random variables $\{ y(s),\, s \neq t\}$; i.e.,
$\E\{y(t)\,\overline{d(s)}\}= \sigma^2 \, \delta_{ts}$, $t, s \,\in\,\Zbb_{2N}$,
where $\delta$ is the Kronecker function and $\sigma^2$ is a   positive number, or, equivalently,
\begin{equation}
\label{ }
\E\{\yb\db^*\}=\sigma^2 \Ib,
\end{equation}
where $\Ib$ denotes the $2N\times 2N$ identity matrix. Interpreting \eqref{finconjn} in the $\mod 2N$ arithmetics of $\Zbb_{2N}$,    $\yb$ admits a linear representation of the form $\Fb \,\yb= \db$,
where  $\Fb$ is a $2N\times 2N$ circulant matrix  with ones on the main diagonal. Following Masani \cite{Masani-60},  $d$ is called the (unnormalized) {\em conjugate process} of $y$. Therefore, setting $\eb := \frac{1}{\sigma^2}  \db$ and $\Ab  :=  \frac{1}{\sigma^2}\Fb$, we see that a full-rank stationary periodic process admits a normalized representation
\begin{equation}\label{Repres}
\Ab \yb= \eb,\quad \E\{\eb \yb^*\} = \Ib, 
\end{equation}
where $\Ab$ is Hermitian and circulant.  Since $\Ab\E\{\yb\yb^*\}=  \E\{\eb\yb^*\} =  \Ib$, $\Ab$ is also positive definite and the covariance matrix \eqref{Sigma} is  given by
\begin{equation}
\label{cov}
\Sigmab   = \Ab ^{-1},
\end{equation}
which is circulant, since the inverse of a circulant matrix is itself circulant. Therefore, if 
\begin{equation}
\label{ }
c_k:=\E\{y(t+k)\overline{y(t)}\},\quad  k=0,1, 2,\dots, N,
\end{equation}
$\Sigmab$ is precisely the Hermitian  circulant matrix 
\begin{equation}
\label{Sigmacirc}
\Sigmab=\Circ\{ c_0,\bar{c}_1,\bar{c}_2,\dots, \bar{c}_N,c_{N-1},\dots,c_2,c_1\}.
\end{equation} 
In fact, a  stationary process $\yb$  is full-rank periodic in $\Zbb_{2N}$, if and only if $\Sigmab $ is a Hermitian  positive  definite circulant matrix \cite{Carli-FPP}. 

We are now in a position to state the main problem of this paper. Supposing that only the covariance lags  $c_0,c_1,\dots,c_n$ are available for $n<N$, how do we complete the matrix   \eqref{Sigmacirc} with the entries $c_{n+1}, c_{n+2}, \dots,c_N$ so that it is circulant and the covariance matrix \eqref{Sigma} of a stationary periodic process $\yb$ with a spectral density of the rational form \eqref{Phi}. We would like to parametrize the set of  all solutions to this problem.

This can be seen as a generalization of  modeling of {\em reciprocal processes\/} about which there is a large and important literature \cite{Jamison-70,Jamison-74,Krener-86,Krener-86,Krener-86b,Frezza-90,Levy-92,Levy-F-02,Levy-F-K-90}. A first step in this direction was taken in \cite{Carli-FPP},  
where the circulant matrix $\Ab$ in \eqref{cov} is required to be {\em banded of order $n$}; i.e., 
\begin{equation}
\label{ }
\Ab = \Circ\{ a_0,\bar{a}_1,\dots,\bar{a}_n,0,\dots,0,a_n,a_{n-1},\dots,a_1\}.
\end{equation}
For example, a banded matrix of order $n=2$ takes the form
\begin{equation}
\Ab = \begin{bmatrix}
		a_0 & a_1& a_2 &0&\cdots& 0 &\bar{a}_2& \bar{a}_1\\ 
		\bar{a}_1 & a_0&  a_1&a_2 &0&\cdots&0  &\bar{a}_2\\ 
		\bar{a}_2&\bar{a}_1&a_0&a_1&a_2&0&\cdots&0 \\
		0 & \bar{a}_2&\bar{a}_1&a_0&a_1&a_2&0&\cdots \\ 
		\vdots & \ddots&\ddots&\ddots &\ddots&\ddots&\ddots&\vdots\\ 
		0&\ddots& 0&\bar{a}_2& \bar{a}_1& a_0&a_1&a_2\\
		a_2&0&\ddots& 0&\bar{a}_2& \bar{a}_1& a_0&a_1\\
		a_{1} & a_2&0&\cdots& 0&\bar{a}_2& \bar{a}_1& a_0\end{bmatrix} .
\end{equation}
In this case $\yb$ admits a bilateral AR representation
\begin{equation}\label{AR-Repres}
\sum _{k=-n}^{n} a_k y(t-k)= e(t)\,,\quad   a_{-k}= \bar{a}_k 
\end{equation}
for all $t \in \Zbb_{2N}$, which can also be written
\begin{equation}
\label{ }
A(\zeta)y(t)=e(t)
\end{equation}
in terms of the symbol 
\begin{equation}
\label{ }
A(\zeta)=\sum_{k=-n}^n a_k \zeta^{-k}, \quad a_k=\bar{a}_k
\end{equation}
of $\Ab$, where $\zeta$ represents the forward shift on $\Zbb_{2N}$.	These models are representations of stationary {\em reciprocal processes of order $n$}; see \cite{Carli-FPP}, where they are determined by solving a  {\em maximum-entropy\/} problem. 

In this paper we show that for each choice of banded circulant matrix $\Pb$  of order at most $n$, there is an unique banded circulant matrix $\Qb$ of order $n$ such that 	    
\begin{equation}
\label{ }
\Sigmab=\Qb^{-1}\Pb. 
\end{equation}
If the corresponding symbols are $P(\zeta)$ and $Q(\zeta)$, respectively, then, by \eqref{Repres} and     \eqref{cov},  such a solution corresponds to a bilateral ARMA representation
\begin{equation}
\label{ARMA}
Q(\zeta)y(t)=P(\zeta)e(t),
\end{equation}
or, equivalently,
\begin{equation}
\label{ }
\sum _{k=-n}^{n} q_k y(t-k)= \sum _{k=-n}^{n} p_k e(t-k).
\end{equation}
We have therefore a complete parameterization of such representations, and hence of the completions of $\Sigmab$, in terms of the $\Pb$.

In Section~\ref{prelsec}, we review basic facts about circulant matrices and harmonic analysis 
on $\mathbb{Z}_{2N}$ and set up notations. The main results on the complete parameterization of the circulant rational covariance extension problem are presented in Section~\ref{mainsec}, where we also consider the circulant rational covariance extension problem as an approximation procedure for the ordinary rational covariance extension problem.  In Section~\ref{logsec}, following \cite{BEL1,BEL2,PEthesis,PE,Georgiou-L-03}, we show how the parameter $\Pb$ can be determined from logarithmic moments computed from data. 

\section{Preliminaries}\label{prelsec} 

Most of the harmonic analysis of stationary processes on $\Zbb$   carries over naturally, provided the Fourier transform is understood as a mapping from functions defined on $\Zbb_{2N}$ onto complex-valued functions on the   unit circle  of the complex plane, regularly sampled at intervals of length  $\Delta:= \pi/N$.  We shall call this object  the {\em discrete unit circle\/} and denote it by $\Tbb_{2N}$. This Fourier map is usually called the {\em discrete Fourier transform (DFT)}. Next we shall review some pertinent facts and at the same time set up notations.

\subsection{Harmonic analysis on $\mathbb{Z}_{2N}$} 

Let     $\zeta_1:=e^{i\Delta} $ be  the primitive $2N$-th root of unity; i.e., $\Delta=\pi/N$, and define the discrete variable $\zeta$ taking the $2N$  values $ \zeta_{k}\equiv \zeta_{1}^{k}=  e^{i\Delta k}\,;\, k=-N+1,\ldots,0,\ldots,N$  running counterclockwise on the  discrete unit circle $\Tbb_{2N}$.   In particular, we have $\zeta_{-k}=\overline{\zeta_k}$ (complex conjugate).\

The discrete Fourier transform $\script F$ maps a finite signal   $g=\{ g_k; \,k= -N+1,\,\dots ,\, N\}$, into a  sequence  of complex numbers 
\begin{equation} \label{DFT}
G(\zeta_{j}) := \sum_{k=-N+1 }^{N }\, g_k \zeta_{j}^{-k} \,,\qquad j=-N+1,-N+2, \ldots , N.
\end{equation}
It is well-known that  the signal $g$ can be recovered from its DFT $G$  by the formula
\begin{equation} \label{InvDFT}
g_k = \sum_{j=-N+1 }^{N  }\zeta_{j}^k G(\zeta_j) \frac{\Delta}{2\pi} \,,\quad k =  -N+1,-N+2, \dots ,N,
\end{equation}
 where $\frac{\Delta}{2\pi}=\frac{1}{2N}$ plays the role of  a uniform discrete measure with  total mass one on the discrete unit circle $\Tbb_{2N}$. In the sequel it will be useful to write \eqref{InvDFT} as an integral
\begin{equation} \label{InvDFTmeas}
g_k = \int_{-\pi}^\pi e^{ik\theta}G(e^{ik\theta})  d\nu(\theta),\quad k =  -N+1,-N+2, \dots ,N,
\end{equation}
where  $\nu$ is a step function with steps $\frac{1}{2N}$ at each $\zeta_k$; i.e., 
\begin{equation}
\label{nu}
d\nu(\theta) =\sum_{-N+1}^N\delta(e^{i\theta}-\zeta_k)\frac{d\theta}{2N}.
\end{equation}
It easy to check that, if $F$ is the DFT of $\{ f_k\}$, 
\begin{equation}
\label{Plancherel}
\sum_{k=-N+1}^N f_k\bar{g}_k=\frac{1}{2N}\sum_{k=-N+1}^N F(\zeta_k)G(\zeta_{-k})= \int_{-\pi}^\pi F(e^{ik\theta}) G(e^{ik\theta})^* d\nu(\theta).
\end{equation}
This is {\em Plancherel's Theorem\/} for DFT.

It is sometimes convenient to write the discrete Fourier transform \eqref{DFT} in the matrix form 
\begin{equation}
\label{ }
\hat{\gb} =\Fb\gb,
\end{equation}
where $\hat{\gb}:=\big(G(\zeta_{-N+1}),G(\zeta_{-N+2}),\dots,G(\zeta_N)\big)\Tr$, $\gb:=(g_{-N+1},g_{-N+2},\dots,g_N)\Tr$ and $\Fb$ is the nonsingular $2N\times 2N$ Vandermonde matrix
\begin{equation}
\label{ }
\Fb= \left[\begin{array}{cccc}\zeta_{-N+1}^{N-1} & \zeta_{-N+1}^{N-2} & \cdots & \zeta_{-N+1}^{-N} \\ \vdots & \vdots & \cdots & \vdots \\\zeta_{0}^{N-1} & \zeta_{0}^{N-2} & \cdots & \zeta_{0}^{-N}  \\ \vdots & \vdots & \cdots & \vdots \\\zeta_{N}^{N-1} & \zeta_{N}^{N-2} & \cdots & \zeta_{N}^{-N} \end{array}\right] .
\end{equation}
Likewise, it follows from \eqref{InvDFT} that
\begin{equation}
\label{ }
\gb=\frac{1}{2N}\Fb^*\hat{\gb},
\end{equation}
i.e., ${\script F}^{-1}$ corresponds to $\frac{1}{2N}\Fb^*$. Consequently, $\Fb\Fb^*=2N\, \Ib$, and hence $\Fb^{-1}=\frac{1}{2N}\Fb^*$ and $(\Fb^*)^{-1}=\frac{1}{2N}\Fb$.
\subsection{Circulant matrices}

A Hermitian circulant matrix
\begin{equation}
\label{Cb}
\Mb:=\Circ\{ m_0,\bar{m}_1,\bar{m}_2,\dots, \bar{m}_N,m_{N-1},\dots,m_2,m_1\}.
\end{equation}
can be represented in the form
\begin{equation}
\label{S2C}
\Mb =\sum_{k=-N+1}^N m_k\Sb^{-k}, \quad m_{-k}=\bar{m}_k
\end{equation}
where $\Sb$ is the nonsingular $2N\times 2N$  cyclic shift matrix
\begin{equation}
\label{Sb}
\Sb := \left[\begin{array}{cccccc}0 & 1 & 0 & 0 & \dots & 0 \\0 & 0 & 1 & 0 & \dots & 0 \\0 & 0 & 0 & 1 & \dots & 0 \\\vdots & \vdots & \vdots & \ddots & \ddots & \vdots \\0 & 0 & 0 & 0 & 0 & 1 \\1 & 0 & 0 & 0 & 0 & 0\end{array}\right] ,
\end{equation}
which is itself a circulant matrix with symbol $S(\zeta)=\zeta$. 
Clearly $\Sb^{2N}=\Sb^0=\Ib$, and 
\begin{equation}
\label{Sbpowers}
\Sb^{k+2N}=\Sb^k, \qquad \Sb^{2N-k}=\Sb^{-k}=(\Sb^k)\Tr .
\end{equation}
Consequently,
\begin{equation}
\label{circulantcondition}
\Sb\Mb \Sb^*=\Mb,
\end{equation}
and the condition \eqref{circulantcondition} is both necessary and sufficient for $\Mb$ to be circulant. (Clearly, what is said in this section holds for circulant matrices in general, but in this paper we are only interested in the Hermitian ones.)

As before setting $\gb:=(g_{-N+1},g_{-N+2},\dots,g_N)\Tr$, we have
\begin{equation}
\label{ }
[\Sb\gb]_k=g_{k+1}, \quad k\in\mathbb{Z}_{2N}.
\end{equation}
In view of \eqref{DFT}, it  then follows that $\zeta {\script F}(\gb)(\zeta)={\script F}(\Sb\gb)(\zeta)$,
from which we have 
\begin{equation}
\label{Cg}
{\script F}(\Mb\gb)(\zeta)=M(\zeta){\script F}(\gb)(\zeta),
\end{equation}
where the pseudo-polynomial
\begin{equation}
\label{symbol}
M(\zeta)=\sum_{k=-N+1}^N m_k \zeta^{-k}
\end{equation}
is called the  {\em symbol\/} of the circulant matrix $\Mb$.

An important property of circulant matrices is that they are diagonalized by the discrete Fourier transform. More precisely,  it follows from \eqref{Cg} that  
\begin{equation}
\label{ }
\Mb=\frac{1}{2N}\Fb^*\text{\rm diag}\big(M(\zeta_{-N+1}),\dots,M(\zeta_{-1}),M(\zeta_0),M(\zeta_1),\dots,M(\zeta_N)\big)\Fb,
\end{equation}
and consequently the inverse $\Mb^{-1}$ is
\begin{equation}
\label{ }
\Mb^{-1}=\frac{1}{2N}\Fb^*\text{\rm diag}\big(M(\zeta_{-N+1})^{-1},\dots,M(\zeta_N)^{-1}\big)\Fb.
\end{equation}
Since 
\begin{equation*}
\label{ }
\Sb =\frac{1}{2N}\Fb^*\text{\rm diag}\big(\zeta_{-N+1},\dots,\zeta_N\big)\Fb\quad\text{and}\quad \Sb^* =\frac{1}{2N}\Fb^*\text{\rm diag}\big(\zeta_{-N+1}^{-1},\dots,\zeta_N^{-1}\big)\Fb,
\end{equation*}
we have
\begin{displaymath}
\Sb\Mb^{-1}\Sb^*=\Mb^{-1}.
\end{displaymath}
Consequently, $\Mb^{-1}$ is also a circulant matrix with symbol $M(\zeta)^{-1}$. In general, in view of the circulant property \eqref{S2C} and \eqref{Sbpowers}, quotients of symbols are themselves pseudo-polynomials of degree at most $N$ and hence symbols. The coefficients of the corresponding   pseudo-polynomial $M(\zeta)^{-1}$ can be determined by Lagrange interpolation. More generally, if $\Ab$ and $\Bb$ are circulant matrices of the same dimension with symbols $A(\zeta)$ and $B(\zeta)$ respectively, then $\Ab\Bb$ and $\Ab+\Bb$ are circulant matrices with symbols $A(\zeta)B(\zeta)$ and $A(\zeta)+B(\zeta)$, respectively. In fact, the circulant matrices of a fixed dimension form an algebra and the DFT is an {\em algebra   homomorphism} of the set of  circulant matrices  onto the   pseudo-polynomials of degree at most $N$ in the variable $\zeta \in \Tbb_{2N}$. 

\subsection{Spectral representation of periodic stationary stochastic processes}

Let $\{y(t)\}$ be a zero-mean  stationary process defined on a finite interval $[-N+1,\,N]$ of the integer line $\Zbb$ and extended to all of $\Zbb$ as a periodic stationary process with  period $2N$. Let $c_{-N+1},c_{-N+2},\dots,c_{N}$ be the covariance lags $c_k:=\E\{ y(t+k)\overline{y(t)}\}$, and define its discrete Fourier transformation
\begin{equation} \label{c2Phi}
\Phi(\zeta_{j}) := \sum_{k=-N+1 }^{N }\, c_k \zeta_{j}^{-k} \,,\qquad j=-N+1,\dots , N,
\end{equation}
which is a positive real-valued function of $\zeta$.
Then, as seen from \eqref{InvDFT} and \eqref{InvDFTmeas}, 
\begin{equation} \label{Phi2c}
c_k = \sum_{j=-N+1 }^{N  }\zeta_{j}^k \Phi(\zeta_j) \frac{\Delta}{2\pi}=\int_{-\pi}^\pi e^{ik\theta}\Phi(e^{ik\theta})  d\nu(\theta) \,,\quad k =  -N+1, \dots ,N.
\end{equation}
The function $\Phi$ is the {\em spectral densitiy\/} of the process $y$. In fact, let
\begin{equation}
\label{yDFT}
\hat{y}(\zeta_k):= \sum_{t=-N+1}^N y(t)\zeta_k^{-t}, \quad k=-N+1,\dots, N,
\end{equation}
be the discrete Fourier transformation of the process $y$. The random variables \eqref{yDFT} turn out to be uncorrelated, and 
\begin{equation}
\label{yhatyhat}
\frac{1}{2N}\E\{ \hat{y}(\zeta_k)\hat{y}(\zeta_\ell)^*\}=\Phi(\zeta_{k})\delta_{k\ell}.
\end{equation}
This can be seen by a straight-forward calculation noting that 
\begin{equation}
\label{orthogonality}
\frac{1}{2N}\sum_{t=-N+1}^N (\zeta_k\zeta_\ell^*)^t =\delta_{k\ell}.
\end{equation}
Then, we obtain a spectral representation of $\yb$ analogous to the usual one, namely
\begin{equation}
\label{ }
y(t)=\sum_{k=-N+1}^N \zeta_k^t\,\hat{y}(\zeta_k)\frac{1}{2N}=\int_{-\pi}^\pi e^{ik\theta}d\hat{y}(\theta),
\end{equation}
where 
\begin{equation}
\label{ }
d\hat{y}(\theta):=\hat{y}(e^{i\theta})d\nu(\theta).
\end{equation}
It is interesting to note that 
\begin{equation}
\label{ }
\Phi(\zeta_{k})=\frac{1}{2N}E\{ \hat{y}(\zeta_k)\hat{y}(\zeta_k)^*\}, \quad k=-N+1,\dots, N, 
\end{equation}
obtained from \eqref{yhatyhat} has a similar form as the {\em periodogram\/}
\begin{equation}
\label{ }
\hat{\Phi}(\zeta_{k})=\frac{1}{2N} \hat{y}(\zeta_k)\hat{y}(\zeta_k)^*, \quad k=-N+1,\dots, N,
\end{equation}
widely used in statistics.

 \section{The complete solution to the circulant rational covariance extension problem}\label{mainsec}

Given $n<N$ and $c_0,c_1,\dots,c_n$ with a positive definite Toeplitz matrix \eqref{Toeplitz}, find a spectral density $\Phi$ satisfying the moment conditions
\begin{equation}
\label{momentconditions}
\int_{-\pi}^\pi e^{ik\theta}\Phi(e^{i\theta})d\nu=c_k, \quad k=0,1,2,\dots,n.
\end{equation}
Note that this moment condition can also be written as an underdetermined system of linear equations
\begin{equation}
\label{ }
\frac{1}{2N}\sum_{j=-N+1}^N \zeta_j^k\,\Phi(\zeta_j)=c_k \quad k=0,1,2,\dots,n,
\end{equation}
in the variables $x_j=\Phi(\zeta_j)$, $j=-N+1,\dots, -1,0,1,\dots,N$, where the coefficient matrix is a Vandermonde matrix of full rank. Note that it is consistent with \eqref{momentconditions} to define negative moments by setting $c_{-k}=\bar{c}_k$, so that the pseudo-polynomial 
\begin{equation}
\label{C(z)}
C(\zeta)=\sum_{k=-n}^n c_k \zeta^{-k}
\end{equation}
 is the symbol of a banded Hermitian circulant matrix
\begin{equation}
\label{bandedC}
\Cb = \Circ\{ c_0,\bar{c}_1,\dots,\bar{c}_n,0,\dots,0,c_n,c_{n-1},\dots,c_1\}
\end{equation}
of order $n$. We would like to find a {\em rational\/} extension $c_{n+1},c_{n+2},\dots,c_N$ replacing the zeros in $\Cb$ to obtain a Hermitian circulant matrix 
\begin{equation}
\label{Sigmab}
\Sigmab:=\Circ\{ c_0,\bar{c}_1,\bar{c}_2,\dots, \bar{c}_N,c_{N-1},\dots,c_2,c_1\}
\end{equation}
that is positive definite. In terms of stationary periodic processes this is the covarance matrix \eqref{Sigma}.  We proceed to solve this in terms of the symbols, and then interpret the results in terms of matrices. 

\subsection{Circulant rational covariance extension in terms of symbols}

Let $\mathfrak{P}$ be the finite-dimensional space of symmetric trigonometric polynomials \eqref{P(z)},
and define $\mathfrak{P}_+$ to be the positive cone 
\begin{displaymath}
\mathfrak{P}_+=\{ P\in\mathfrak{P}\mid P(e^{i\theta})> 0 \quad \text{\rm for all}\; \theta\in [-\pi,\pi]\}.
\end{displaymath}
Moreover, let $\mathfrak{C}_+$ be the dual cone of all $\cb=(c_0,c_1,\dots,c_n)$ such that 
\begin{equation}
\label{ }
\langle C,P\rangle := \sum_{k=-n}^n c_k\bar{p}_k \geq 0 \quad \text{\rm for all} \; P\in\overline{\mathfrak{P}_+},
\end{equation}
where the notation  $\langle C,P\rangle$ is motivated by the fact that, by \eqref{Plancherel},  
\begin{equation}
\label{<c,p>int}
\langle C,P\rangle =\int_{-\pi}^\pi C(e^{i\theta})P(e^{i\theta})^*d\nu.
\end{equation}
It can be shown that $\cb\in\mathfrak{C}_+$ if and only if $\Tb_n> 0$.  In fact, if $a(z)=a_0z^n+\dots +a_{n-1}z +a_n$ is a polynomial spectral factor of $P(z)$, i.e., $a(z)a(z)^*=P(z)$, then it is easy to see that
\begin{equation}
\label{ }
\langle C,P\rangle =\mathbf{a}^* \Tb_n\mathbf{a},
\end{equation}
where $\Tb_n$ is the Toeplitz matrix \eqref{Toeplitz} of  $c_0,c_1,\dots,c_n$ and $\mathbf{a}=(a_0,a_1,\dots,a_n)\Tr$.

Next, define the cone 
\begin{equation}
\label{ }
\mathfrak{P}_+(N)=\big\{ P\in\mathfrak{P}\mid P(\zeta_k)> 0 \quad k=-N+1,-N+2,\dots,N\big\}.
\end{equation}
Clearly, $\mathfrak{P}_+(N)\supset \mathfrak{P}_+(2N)\supset \mathfrak{P}_+(4N)\supset\dots\supset \mathfrak{P}_+$, and the corresponding dual cones satisfy 
\begin{equation}
\label{coneinclusion}
\mathfrak{C}_+(N)\subset \mathfrak{C}_+(2N)\subset \mathfrak{C}_+(4N)\subset\cdots  \subset \mathfrak{C}_+.
\end{equation}

\begin{thm}\label{mainthm}
Let $\cb\in\mathfrak{C}_+(N)$. Then, for each $P\in\mathfrak{P}_+(N)$, there is a unique $Q\in\mathfrak{P}_+(N)$ such that 
\begin{equation}
\label{Phi=P/Q}
\Phi=\frac{P}{Q}
\end{equation}
satisfies the moment conditions \eqref{momentconditions}.
\end{thm}

For the proof, which is given in the appendix, we need to consider a dual pair of  optimization problems. First consider the primal problem to maximize the generalized entropy
\begin{equation}
\label{primal}
\mathbb{I}_P(\Phi) =\int_{-\pi}^\pi  P(e^{i\theta})\log \Phi(e^{i\theta})d\nu
\end{equation}
subject to the moment conditions \eqref{momentconditions}. The corresponding Lagrangian is then given by 
\begin{eqnarray}
L(\Phi,Q) & = & \mathbb{I}_P(\Phi)+\sum_{k=-n}^n \bar{q}_k\left( c_k -\int_{-\pi}^\pi e^{ik\theta}\Phi(e^{i\theta})d\nu\right) \notag\\
      & = & \int_{-\pi}^\pi  P(e^{i\theta})\log \Phi(e^{i\theta})d\nu +\langle C,Q\rangle -  \int_{-\pi}^\pi Q(e^{i\theta})\Phi(e^{i\theta})d\nu,\label{Lagrangian}
\end{eqnarray}
where $q_0,q_1,\dots,q_n$ are Lagrange multipliers, and where $Q$ is defined as in \eqref{P(z)} with $q_{-k}=\bar{q}_k$.  Since  the dual functional $\sup_\Phi L(\Phi,Q)$ is finite only if $Q\in\overline{\mathfrak{P}_+(N)}$, we may restrict the Lagrange multipliers to that set. Therefore, for each $Q\in\overline{\mathfrak{P}_+(N)}$, consider the directional derivative
\begin{displaymath}
\delta L(\Phi,Q; \delta\Phi)=\int_{-\pi}^\pi \left( \frac{P}{\Phi}-Q\right)\delta\Phi d\nu,
\end{displaymath} 
which equals zero for all variations $\delta\Phi$ if and only if 
\begin{displaymath}
\Phi=\frac{P}{Q}.
\end{displaymath}
Inserting this into \eqref{Lagrangian} we obtain
\begin{displaymath}
\sup_\Phi L(\Phi,Q)=\mathbb{J}_P(Q)+\int_{-\pi}^\pi P(e^{i\theta})\left[\log P(e^{i\theta})-1\right]d\nu,
\end{displaymath}
where 
\begin{equation}
\label{dual}
\mathbb{J}_P(Q)= \langle C,Q\rangle -\int_{-\pi}^\pi  P(e^{i\theta})\log Q(e^{i\theta})d\nu
\end{equation}
and the last term is constant. Hence we may take \eqref{dual} as the dual functional. 

It will be shown below that $\mathbb{J}_P$ is strictly convex, so a stationary point in $\mathfrak{P}_+$, if it exists, would have to be a unique minimizer of $\mathbb{J}_P$. For $k=1,2,\dots,n$, we write $q_k=x_k+y_k$ as a sum of real and imaginary parts and define the partial differential operators 
\begin{subequations}
\begin{align}
\label{partial diff}
 \frac{\partial\phantom{x}}{\partial q_k} &=\frac12\left( \frac{\partial\phantom{x}}{\partial x_k} -i \frac{\partial\phantom{x}}{\partial y_k}\right) \\
 \frac{\partial\phantom{x}}{\partial \bar{q}_k} &=\frac12\left( \frac{\partial\phantom{x}}{\partial x_k} +i \frac{\partial\phantom{x}}{\partial y_k}\right)
 \end{align}
\end{subequations}
in the standard way; see, e.g., \cite[p. 1]{Hormander}. It is immediately seen that 
\begin{equation}
\label{dq/dqbar}
 \frac{\partial q_k}{\partial \bar{q}_k}=0 \quad \text{and}\quad  \frac{\partial\bar{q}_k}{\partial q_k}=0. 
\end{equation}
From this, we readily obtain
\begin{equation}
\label{Jgradient}
\frac{\partial\mathbb{J}_P}{\partial  \bar{q}_k}= 
c_k -\int_{-\pi}^\pi e^{ik\theta}\frac{P(e^{i\theta})}{Q(e^{i\theta})}d\nu, \quad k=1,2,\dots,n.
\end{equation}
Setting \eqref{Jgradient} equal to zero yields the moment conditions \eqref{momentconditions}.
Then the proof of the following theorem follows directly from Theorem~\ref{mainthm}.

\begin{thm}\label{optthm}
Let Let $\cb\in\mathfrak{C}_+(N)$ and $P\in\mathfrak{P}_+(N)$. Then  the problem to maximize \eqref{primal}  subject to the moment conditions \eqref{momentconditions} has a unique solution, namely \eqref{Phi=P/Q}, where $Q$ is the unique optimal solution of the problem to minimize \eqref{dual}
over all $Q\in\mathfrak{P}_+(N)$. 
\end{thm}

From \eqref{Jgradient} we have the Hessian 
\begin{equation}
\label{Hessian}
\frac{\partial^2\mathbb{J}_P}{\partial \bar{q}_k\partial q_\ell}= \int_{-\pi}^\pi e^{i(k-\ell)\theta}\frac{P(e^{i\theta})}{Q(e^{i\theta})^2}d\nu,\quad k,\ell=0,1,\dots,n,
\end{equation}
which is Hermitian and positive definite, showing that $\mathbb{J}_P$ is strictly convex. 

Next, we establish that the solution to the circulant rational covariance extension problem depends smoothly on  the parameters $\cb$ and $P$. To this end, for each fixed $P\in\mathfrak{P}_+(N)$, we define the moment map $F^P$ componentwise given by
\begin{equation}
\label{momentmap}
F^P_k(Q)=\int_{-\pi}^\pi e^{ik\theta}\frac{P(e^{i\theta})}{Q(e^{i\theta})}d\nu, \quad k=0,1,\dots,n
\end{equation}
and, for each $c\in\mathfrak{C}_+$, the map $G^c$ sending $P\in\mathfrak{P}_+(N)$ to $Q\in\mathfrak{Q}_+(N):=(F^P)^{-1}(\mathfrak{C}_+(N))$.

The proof of the following theorem is given in the Appendix.

\begin{thm}\label{homeomorphism}
The maps $F^P$ and $G^c$ are homeomorphisms.
\end{thm}

In particular, we have established a complete smooth parameterization of all solutions $Q$ to the circulant rational covariance extension problem in terms of $P\in\mathfrak{P}_+(N)$. 

Finally, as a prerequisite for Theorem~\ref{PQthm} in the next section, we show that the map $F^P$ can be continuously extended to the boundary $\mathfrak{P}_+(N)$, as can be seen from the following extension, proved in the Appendix, of the family of dual solutions.

\begin{thm}\label{dualboundarythm}
Let $\cb\in\mathfrak{C}_+(N)$, and let $\overline{\mathfrak{P}_+(N)}$ denote the closure of $\mathfrak{P}_+(N)$. Then, for each $P\in\overline{\mathfrak{P}_+(N)}\setminus \{0\}$, the dual problem to minimize \eqref{dual} over $Q\in\overline{\mathfrak{P}_+(N)}\setminus \{0\}$ has a unique minimizer  $\hat{Q}$, and $P/\hat{Q}$ satisfies the moment conditions \eqref{momentconditions}.
\end{thm}

\subsection{Circulant rational covariance extension in terms of matrices}

Next we reformulate the optimization problems in terms of circulant matrices. To this end, we define the circulant matrix
\begin{equation}
\label{ }
\Sigmab=\frac{1}{2N}\Fb^*\text{\rm diag}\big(\Phi(\zeta_{-N+1}),\dots,\Phi(\zeta_N)\big)\Fb
\end{equation}
with symbol \eqref{Phi=P/Q} and the banded numerator matrix 
\begin{equation}
\label{ }
\Pb=\frac{1}{2N}\Fb^*\text{\rm diag}\big(P(\zeta_{-N+1}),\dots,P(\zeta_N)\big)\Fb
\end{equation}
of degree at most $n$ with symbol \eqref{P(z)}. Since $\Phi(\zeta_k)>0$ for all $k$ and  $\log\Phi(\zeta)$ is analytic in the neighborhood of each $\Phi(\zeta_k)>0$, by the spectral mapping theorem \cite[p. 557]{DS} the eigenvalues of $\log\Sigmab$ are just the real numbers $\log\Phi(\zeta_k)$, $k=-N+1,\dots,N$, and hence 
\begin{equation}
\label{}
\log\Sigmab=\frac{1}{2N}\Fb^*\text{\rm diag}\big(\log\Phi(\zeta_{-N+1}),\dots,\log\Phi(\zeta_N)\big)\Fb.
\end{equation}
Consequently, the primal functional \eqref{primal} may be written
\begin{eqnarray}
\int_{-\pi}^\pi  P(e^{i\theta})\log \Phi(e^{i\theta})d\nu & = & \frac{1}{2N}\sum_{j=-N+1}^N P(\zeta_j)\log\Phi(\zeta_j)\notag\\
& = &\frac{1}{2N} \text{tr}\{\Pb\log\Sigmab\} 
\end{eqnarray}
and the moment conditions \eqref{momentconditions} as 
\begin{equation}
\label{momentcondmatrixa}
\frac{1}{2N} \text{tr}\{\Sb^k\Sigmab\} =c_k, \quad k=0,1,\dots,n,
\end{equation}
or, equivalently,  as 
\begin{equation}
\label{momentcondmatrixb}
\Eb_n\Tr\Sigmab \Eb_n =\Tb_n, \quad \text{where\;} \Eb_n =\begin{bmatrix}\Ib_n\\{\bold 0}\end{bmatrix}.
\end{equation}
Consequently the primal problem amounts to maximizing $\text{tr}\{\Pb\log\Sigmab\}$ over all Hermitian, positive definite $2N\times 2N$ matrices subject to \eqref{momentcondmatrixa} or   \eqref{momentcondmatrixb}. For the special case $P\equiv 1$ 
this reduces to the primal problem presented in \cite{Carli-FPP}, except that in  \cite{Carli-FPP} there is an extra condition insuring that $\Sigmab$ is circulant. However, it was shown in \cite{CarliGeorgiou} that this condition is automatically satisfied and is therefore not needed.

In the same way, by \eqref{<c,p>int} the dual functional \eqref{dual} can be written
\begin{equation}
\label{matrixdual}
\begin{split}
\int_{-\pi}^\pi C(e^{i\theta})Q(e^{i\theta})d\nu-\int_{-\pi}^\pi  P(e^{i\theta})\log Q(e^{i\theta})d\nu\\
=\frac{1}{2N}\text{tr}\{\Cb\Qb\} -\frac{1}{2N}\text{tr}\{\Pb\log\Qb\},
\end{split}
\end{equation}
where 
\begin{equation}
\label{ }
\Qb=\frac{1}{2N}\Fb^*\text{\rm diag}\big(Q(\zeta_{-N+1}),\dots,Q(\zeta_N)\big)\Fb
\end{equation}
and $\Cb$ is the banded circulant matrix \eqref{bandedC} formed from $c_0,c_1,\dots,c_n$. 

Consequently, given  $\cb\in\mathfrak{C}_+(N)$, it follows from Theorem~\ref{mainthm} that, for each  Hermitian, positive circulant matrix $\Pb$ that  is banded of degree at most $n$, 
there is a unique $\Sigmab$ given by
\begin{equation}
\label{ }
\Sigmab =\Qb^{-1}\Pb,
\end{equation}
where $\Qb$ is  the unique solution of the problem to minimize 
\begin{equation}
\label{ }
\mathbb{J}_\Pb(\qb)=\frac{1}{2N}\text{tr}\{\Cb\Qb\} -\frac{1}{2N}\text{tr}\{\Pb\log\Qb\}
\end{equation}
over all  $\qb:=(q_0,q_1,\dots,q_n)$ such that the Hermitian, circulant matrix
$$\Qb = \Circ\{ q_0,\bar{q}_1,\dots,\bar{q}_n,0,\dots,0,q_n,q_{n-1},\dots,a_1\}$$
is positive definite. For the maximum-entropy solution corresponding to $\Pb=\Ib$ this reduces to an optimization problem that is different from the one presented in \cite{Carli-FPP}.

Since
\begin{displaymath}
\Qb=\sum_{k=-n}^n q_k\Sb^{-k}, \quad q_{-k}=\bar{q}_k,
\end{displaymath}
we have 
\begin{displaymath}
\frac{\partial\Qb}{\partial q_k}=\Sb^{-k}\quad \text{and}\quad \frac{\partial\Qb}{\partial \bar{q}_k}=\Sb^k,
\end{displaymath}
for $k=0,1,\dots,n$, and therefore
\begin{equation}
\label{gradmatrtix}
\frac{\partial\Jb_\Pb}{\partial \bar{q}_k}=\frac{1}{2N}\text{tr}\{\Sb^k\Cb\} -\frac{1}{2N}\text{tr}\{\Sb^k\Pb\Qb^{-1}\}=c_k-\frac{1}{2N}\text{tr}\{\Sb^k\Pb\Qb^{-1}\},
\end{equation}
where we have used the fact that 
\begin{displaymath}
\text{tr}\{\Sb^k\Cb\}=\sum_{j=-n}^n c_j\text{tr}\{\Sb^{k-j}\}=2N c_k,
\end{displaymath}
as $\text{tr}\{\Sb^k\}\ne 0$ only for $\Sb^0=\Ib$. Setting \eqref{gradmatrtix} equal to zero yields the moment conditions. Likewise, 
\begin{equation}
\label{Hessianmatrix}
\frac{\partial^2\Jb_\Pb}{\partial \bar{q}_k\partial q_k}= \frac{1}{2N}\text{tr}\{\Sb^k\Pb\Qb^{-2}\Sb^{-\ell}\}=\frac{1}{2N}\text{tr}\{\Sb^{k-\ell}\Pb\Qb^{-2}\},
\end{equation}
showing that the Hessian is a Toeplitz matrix. This is the matrix version of \eqref{Hessian}.

In \cite{Carli-FPP} it was observed that the condition that the Toeplitz matrix $\Tb_n$, defined by \eqref{Toeplitz}, is positive definite is a necessary, but not a sufficient, condition for feasibility of the circulant banded covariance extension problem. This can now be understood in the more general setting of moment problems discussed above. In fact, the Toeplitz condition $\Tb_n>0$ is equivalent to $\cb\in\mathfrak{C}_+$. whereas, by Theorem~\ref{optthm},  $\cb\in\mathfrak{C}_+(N)$ is required for feasibility. Since $\mathfrak{C}_+(N)\subset\mathfrak{C}_+$, it follows that the Toeplitz condition cannot be sufficient in general. However, as proved in 
\cite{Carli-FPP}, feasibility is achieved for a sufficiently large $N$. This can also be seen from the following result. 

\begin{prop}\label{feasibilityprop}
The feasibility set $\mathfrak{C}_+(N)\to\mathfrak{C}_+$ as $N\to\infty$. In particular, 
for any $\cb\in\mathfrak{C}_+$, there is an $N_0$ such that $\cb\in\mathfrak{C}_+(N)$ for $N\geq N_0$.   
\end{prop}

\begin{proof}
As  $N\to\infty$, the set $\{\zeta_j;\; j=-N+1,\dots,N\}$ becomes dense on the unit circle, and therefore $\mathfrak{P}_+(N)\to \mathfrak{P}_+$. Consequently, $\mathfrak{C}_+(N)\to\mathfrak{C}_+$, and the convergence is monotone in the sense of \eqref{coneinclusion}. Therefore, since $\mathfrak{C}_+$ is an open set, there is an $N_0$ such that any $\cb\in\mathfrak{C}_+$ will sooner or later end up in  $\mathfrak{C}_+(N)$ and remain there as $N$ increases.
\end{proof}

\subsection{Some computational considerations}\label{compsubsec}

For each choice of $P$, the Hessian of $\mathbb{J}_P$ can be computed explicitly in terms of  $Q$ as the Toeplitz matrix
\begin{equation}
\label{ }
\Hb(\qb)=\begin{bmatrix} h_0&h_1&h_2&\cdots&h_n\\
				\bar{h}_1&h_0&h_1&\cdots& h_{n-1}\\
				\bar{h}_2&\bar{h}_1&h_0&\cdots&h_{n-2}\\
				\vdots&\vdots&\vdots&\ddots&\vdots\\
				\bar{h}_n&\bar{h}_{n-1}&\bar{h}_{n-2}&\cdots&h_0
 		\end{bmatrix},
\end{equation}
where $\qb=(q_0,q_1,\dots,q_n)$ are the coefficients in the pseudo-polynomial $Q$ and 
\begin{equation}
\label{ }
h_k=\int_{-\pi}^\pi e^{ik\theta}\frac{P(e^{i\theta})}{Q(e^{i\theta})^2}d\nu=
\frac{1}{2N}\sum_{j=-N+1}^N\zeta_j^k\frac{P(\zeta_j)}{Q(\zeta_j)^2},
\end{equation}
as can be seen from \eqref{Hessian} or \eqref{Hessianmatrix}. Therefore
Newton's method can be used to find the unique minimizer of the dual problem. The gradient \eqref{Jgradient} at the point $\qb$ is $\left(\cb-\bar{\cb}(\qb)\right)$, where
\begin{equation}
\label{ }
\bar{\cb}(\qb)=\frac{1}{2N}\sum_{j=-N+1}^N\zeta_j^k\frac{P(\zeta_j)}{Q(\zeta_j)},
\end{equation}
and consequently a Newton step amounts to solving the Toeplitz system
\begin{equation}
\label{ }
\Hb(\qb^{(k)})\left( \qb^{(k+1)}-\qb^{(k)}\right)=\cb-\bar{\cb}(\qb^{(k)}).
\end{equation}
Clearly, $\Hb(\qb)$ and $\bar{\cb}$ can be computed by the discrete Fourier transform.

\subsection{An approximation procedure for the ordinary rational covariance extension problem}

Given a $\cb\in\mathfrak{C}_+$ and a $P\in\mathfrak{P}_+$, the {\em ordinary\/} rational covariance extension problem amounts to finding the unique $Q\in\mathfrak{P}_+$ satisfying the moment conditions
\begin{equation}
\label{contmomentproblem}
c_k =\int_{-\pi}^\pi e^{ik\theta}\frac{P(e^{i\theta})}{Q(e^{i\theta})}\frac{d\theta}{2\pi}, \quad k=0,1,\dots,n
\end{equation}
We would like to approximate the solution $Q$ of this problem by the unique solution $Q_N$ of the circulant rational covariance extension problem 
\begin{equation}
\label{circulantmomentproblem}
c_k =\int_{-\pi}^\pi e^{ik\theta}\frac{P(e^{i\theta})}{Q_N(e^{i\theta})}d\nu_N, \quad k=0,1,\dots,n,
\end{equation}
where $d\nu_N$ is the measure \eqref{nu} corresponding to $N$. 

\begin{thm}
Let $\cb\in\mathfrak{C}_+$ and a $P\in\mathfrak{P}_+$. Moreover, for any $N\geq N_0$, where $N_0$ is defined as in Proposition~\ref{feasibilityprop}, let $Q_N$ be the unique solution of \eqref{circulantmomentproblem}, and let $Q$ be the unique solution of \eqref{contmomentproblem}. Then $Q_N\to Q$ as $N\to\infty$. 
\end{thm}

\begin{proof}
Let $F:\, \mathfrak{P}_+\to\mathfrak{C}_+$ be the map sending $Q$ to $\cb$ as in in \eqref{contmomentproblem}; i.e., $\cb=F(Q)$. Given $Q_N$, define $\cb^{(N)}:=F(Q_N)$ with components
\begin{equation}
\label{Riemannlimit}
c_k^{(N)} =\int_{-\pi}^\pi e^{ik\theta}\frac{P(e^{i\theta})}{Q_N(e^{i\theta})}\frac{d\theta}{2\pi}, \quad k=0,1,\dots,n
\end{equation}
for each $N\geq N_0$.  Since \eqref{circulantmomentproblem} is a Riemann sum converging to \eqref{Riemannlimit} as $N$ in the measure $d\nu_N$ tends to $\infty$ but $Q_N$ is kept fixed,  there is for each $\epsilon>0$ an $N_1\geq N_0$ such that $\|\cb^{(N)}-\cb\|<\epsilon$ for all $N\geq N_1$. Consequently, since  $F(Q_N)=\cb^{(N)}$, $F(Q)=\cb$  and the map $F$ is a diffeomorphism \cite[Theorem 1.3]{BLmoments}, $Q_N\to Q$ as in $N\to\infty$.
\end{proof}

\section{Determining $\Pb$ from logarithmic moments}\label{logsec}

We have shown that the solutions of the circulant rational covariance extension problem are completely parameterized in  a smooth manner by the  numerator trigonometric polynomials $P\in\mathfrak{P}_+(N)$, or, equivalently, by their corresponding banded circulant matrices $\Pb$. Next, we show how $P$ can be determined from the logarithmic moments
\begin{equation}
\label{cepstrum}
m_k=\int_{-\pi}^\pi e^{ik\theta}\log\Phi(e^{i\theta})d\nu, \quad k=1,2,\dots,n.
\end{equation}
In the setting of the classical trigonometric moment problem such moments are known as {\em cepstral coefficients}, and in speech processing, for example, they are estimated for design purposes.

Now consider the problem of finding the spectral density $\Phi$, or, equivalently, the circulant matrix $\Sigmab$, that maximizes the entropy gain
\begin{equation}
\label{entropy}
\mathbb{I}(\Phi) =\int_{-\pi}^\pi \log\Phi(e^{i\theta})d\nu = \frac{1}{2N}\trace\log\Sigmab
\end{equation}
subject to the two sets of moment conditions \eqref{momentconditions} and \eqref{cepstrum}. Such a problem was apparently first considered in the usual trigonometric moment setting in an unpublished technical report \cite{Musicus} and then, independently and in a more elaborate form, in \cite{BEL1,BEL2,PEthesis}. 

Defining 
\begin{equation}
\label{M}
M(\zeta)=\sum_{k=-n}^n m_k \zeta^{-k}, 
\end{equation}
where $m_{-k}=\bar{m}_k$, $k=1,2,\dots,n$, and $m_0=0$, the Langrangian for this optimization problem can be written
\begin{equation}
\begin{split}
L(\Phi,P,Q) = \mathbb{I}(\Phi)+\sum_{k=-n}^n \bar{q}_k\left( c_k -\int_{-\pi}^\pi e^{ik\theta}\Phi(e^{i\theta})d\nu\right) \\- \sum_{k=-n}^n \bar{p}_k\left( m_k -\int_{-\pi}^\pi e^{ik\theta}\log\Phi(e^{i\theta})d\nu\right)\\
     = \langle C,Q\rangle - \langle M,P\rangle  -\int_{-\pi}^\pi Q(e^{i\theta})\Phi(e^{i\theta})d\nu \\
      +\int_{-\pi}^\pi  P(e^{i\theta})\log \Phi(e^{i\theta})d\nu,\label{LagrangianPQ}
      \end{split}
\end{equation}
where $p_1,\dots,p_n,q_0,q_1,\dots,q_n$ are Lagrange multipliers and $p_0:=1$, and $P$ and $Q$ are the corresponding trigonometric polynomials \eqref{P(z)}. For the dual functional $(P,Q)\mapsto \sup_\Phi L(\Phi,P,Q)$ to be finite, $P$ and $Q$ must obviously be restricted to the closure of the cone $\mathfrak{P}_+(N)$. Therefore, for each such choice of $(P,Q)$, we have the directional derivative
\begin{equation}
\label{ }
\delta L(\Phi,P,Q;\delta\Phi) = \int_{-\pi}^\pi \left(\frac{P}{\Phi}-Q\right)\delta\Phi d\nu,
\end{equation}
and hence a stationary point must satisfy 
\begin{equation}
\label{Phi=P/Q2}
\Phi=\frac{P}{Q},
\end{equation}
 which inserted into \eqref{LagrangianPQ} yields
\begin{displaymath}
\sup_\Phi L(\Phi,P,Q)=\mathbb{J}(P,Q)-1,
\end{displaymath}
where 
\begin{equation}
\label{PQdual}
\mathbb{J}(P,Q)= \langle C,Q\rangle -\langle M,P\rangle +\int_{-\pi}^\pi  P(e^{i\theta})\log\frac{P(e^{i\theta})}{Q(e^{i\theta})}d\nu,
\end{equation}
and where we have used the fact that $\int Pd\nu=p_0=1$. Accordingly, we define the bounded subset 
\begin{equation}
\label{ }
\mathfrak{P}^\mathbf{o}_+(N):=\{ P\in\mathfrak{P}_+(N)\mid p_0=1\}.
\end{equation}
of the cone $\mathfrak{P}_+(N)$. Note that $\mathbb{J}$ is convex, but not necessarily strictly convex unless $P$ and $Q$ are coprime, and that 
\begin{subequations}
\begin{eqnarray}
\frac{\partial\mathbb{J}}{\partial \bar{q}_k}&=& c_k -\int_{-\pi}^\pi e^{ik\theta}\frac{P(e^{i\theta})}{Q(e^{i\theta})}d\nu, \quad k=1,\dots, n  \label{gradQ}\\
\frac{\partial\mathbb{J}}{\partial \bar{p}_k} & = & \int_{-\pi}^\pi e^{ik\theta}\log\frac{P(e^{i\theta})}{Q(e^{i\theta})}d\nu -m_k, \quad k=1,\dots,n.\label{gradP}
\end{eqnarray}
\end{subequations}
Consequently, if there exists a stationary point $(P,Q)\in\mathfrak{P}^\mathbf{o}_+(N)\times\mathfrak{P}_+(N)$, \eqref{Phi=P/Q2} will satisfy both the moment conditions \eqref{momentconditions} and \eqref{cepstrum}. 

A proof of the following theorem, which is a circulant version of Theorem~5.3 in \cite{BEL2},  will be given in the Appendix.

\begin{thm}\label{PQthm}
Suppose that $c\in\mathfrak{C}_+(N)$ and $m_1,\dots,m_n$ are complex numbers. Then there exists a solution $(\hat{P},\hat{Q})$ that minimizes $\mathbb{J}(P,Q)$ over all $(P,Q)\in\overline{\mathfrak{P}^\mathbf{o}_+(N)}\times\overline{\mathfrak{P}_+(N)}$, and, for any such solution
\begin{equation}
\label{Phihat2}
\hat{\Phi}=\frac{\hat{P}}{\hat{Q}}
\end{equation}
satisfies the covariance moment conditions \eqref{momentconditions}. If, in addition, $\hat{P}\in\mathfrak{P}_+(N)$, \eqref{Phihat2} also satisfies the logarithmic moment conditions \eqref{cepstrum} and is an optimal solution of the primal problem to maximize the entropy gain \eqref{entropy} given \eqref{momentconditions} and \eqref{cepstrum}. Then $\hat{Q}\in\mathfrak{P}_+(N)$, and the solution is unique. In fact, $\mathbb{J}$ is strictly convex on $\mathfrak{P}^\mathbf{o}_+(N)\times\mathfrak{P}_+(N)$.
\end{thm}

Consequently, solving these optimization problems will always lead to a spectral density with the prescribed covariance lags $c_0,c_1,\dots,c_n$, provided $\cb\in\mathfrak{C}_+(N)$. However, we have not prescribed any condition on the logarithmic moments $m_1,\dots,m_n$, as such a condition is hard to find and would depend on $\cb$. 
If the moments $c_0,c_1,\dots,c_n$ and $m_1,\dots,m_n$ come from the same theoretical spectral density, the optimal solution \eqref{Phihat2} will also match the cepstral coefficients. In practice, however, $c_0,c_1,\dots,c_n$ and $m_1,\dots,m_n$ will be estimated from different data sets, so there is no guarantee that $\hat{P}$ does not end up on the boundary of $\mathfrak{P}_+(N)$ without satisfying the logarithmic moment conditions. Then the problem needs to be regularized, leading to adjusted values of $m_1,\dots,m_n$  consistent with the covariances $c_0,c_1,\dots,c_n$. 

We shall use a regularization term proposed by P. Enqvist \cite{PEthesis} in the context of the usual rational covariance extension problem. More precisely, we consider the regularized dual problem to find a pair $(P,Q)\in \mathfrak{P}_+^\mathbf{o}(N)\times \mathfrak{P}_+(N)$ minimizing
\begin{equation}
\label{regularizeddual}
\mathbf{J}_\lambda(P,Q)=\mathbb{J}(P,Q) -\lambda \int_{-\pi}^\pi \log P(e^{i\theta})d\nu 
\end{equation}
for some $\lambda>0$, or in circulant matrix form
\begin{equation}
\label{Jlambda}
\mathbf{J}_\lambda(P,Q)=\frac{1}{2N}\trace\{\Cb\Qb\} -\frac{1}{2N}\trace\{\Mb\Pb\}+ \frac{1}{2N}\trace\{\Pb\log\Pb\Qb^{-1}\} - \frac{\lambda}{2N}\trace\{\log\Pb\}.
\end{equation}
This functional will take an infinite value for  $P\in\partial\mathfrak{P}_+(N)$, since then $P(\zeta_k)=0$ for some $k$, and hence the minimum will be in the interior.   Then
\begin{equation}
\label{ }
\frac{\partial\mathbf{J}_\lambda}{\partial \bar{p}_k} = \int_{-\pi}^\pi e^{ik\theta}\log\frac{P(e^{i\theta})}{Q(e^{i\theta})}d\nu -m_k -\varepsilon_k =0, \quad k=1,\dots,n,
\end{equation}
at the minimum, where
\begin{equation}
\label{ }
\varepsilon_k=\int_{-\pi}^\pi e^{ik\theta}\frac{\lambda}{P(e^{i\theta})}d\nu=\frac{\lambda}{2N}\sum_{j=-N+1}^N\frac{\zeta_j^k}{P(\zeta_j)} = \frac{\lambda}{2N}\trace\{\Sb^k\Pb^{-1}\}, 
\end{equation}
and hence the moments \eqref{momentconditions} and \eqref{cepstrum} are matched provided one adjusts the logarithmic moments $m_1,m_2,\dots, m_n$ to $m_1+\varepsilon_1,m_2+\varepsilon_2,\dots, m_n+\varepsilon_n$, the latter of which are consistent with $c_0,c_1,\dots,c_n$.  Modifying the analysis in \cite[p. 188 - 196]{PEthesis} to the present setting it is easy to see that \eqref{Jlambda} is a monotonically  nonincreasing function of $\lambda$, and that the solution tends as $\lambda\to\infty$ to a $(\hat{P},\hat{Q})$ where $\hat{P}\equiv 1$, i.e., the maximum entropy solution. 

Computing the Hessian of $\mathbf{J}_\lambda$, we notice that 
\begin{subequations}
\begin{equation}
\frac{\partial \mathbf{J}_\lambda}{\partial \bar{q}_k\partial q_\ell}=
 \frac{1}{2N}\sum_{j=-N+1}^N\zeta_j^{k-\ell}\frac{P(\zeta_j)}{Q(\zeta_j)^2}
\end{equation}
is the same as  \eqref{Hessian}. Moreover,
\begin{align}
\label{}
  \frac{\partial \mathbf{J}_\lambda}{\partial \bar{q}_k\partial p_\ell}  &=  - \frac{1}{2N}\sum_{j=-N+1}^N\zeta_j^{k-\ell}\frac{1}{Q(\zeta_j)} \\
 \frac{\partial \mathbf{J}_\lambda}{\partial \bar{p}_k\partial p_\ell}   &=   \frac{1}{2N}\sum_{j=-N+1}^N\zeta_j^{k-\ell}\frac{1}{P(\zeta_j)} +\frac{1}{2N}\sum_{j=-N+1}^N\zeta_j^{k-\ell}\frac{\lambda}{P(\zeta_j)^2}
\end{align}
\end{subequations}
Since $\mathbb{J}$ is strictly convex (Theorem~\ref{PQthm}), then so is $\mathbf{J}_\lambda$, so the Hessian is positive definite. 
Newton's method can then be used as in Section~\ref{compsubsec} to  determine the unique minimizer. 

\section{Conclusions}

In this paper we have presented a complete parameterization of all solutions to the circulant covariance extension problem. We have shown that determining these solutions involves only computations based on the fast Fourier transform, potentially leading to efficient numerical procedures. This also provides a natural approximation for the ordinary rational covariance extension problem.

The circulant rational covariance extension problem is an inverse problem with infinitely many solutions in general, but by matching additional data in the form of logarithmic moments a unique solution can be determined. 

For many applications it will be important to generalize these results to the multivariable case. This should be straight-forward, but we have chosen to consider only the scalar case in this paper in order  to keep notations reasonably simple and not blur the picture.

\appendix

\subsection{Proof of Theorem~\ref{mainthm}}

Consider the moment map $F^P:\,\mathfrak{P}_+(N)\to \mathfrak{C}_+(N)$ defined by \eqref{momentmap}
for an arbitrary $P\in\mathfrak{P}_+(N)$. This  is a continuous map between connected spaces of  the same (finite) dimension. Therefore, if we can prove that $F^P$ is injective and proper -- i.e., for any compact $K\subset \mathfrak{C}_+(N)$ the inverse image $(F^P)^{-1}(K)$ is compact -- then, by Theorem 2.6 in \cite{BLvariational}, it is a homeomorphism, implying in particular that the system of moment equations $F^P(Q)=\cb$ has a unique solution in $\mathfrak{P}_+(N)$.

\begin{lemma}
The moment map $F^P:\,\mathfrak{P}_+(N)\to \mathfrak{C}_+(N)$ is injective.
\end{lemma}

\begin{proof}
From \eqref{dual} we have the gradient \eqref{Jgradient} 
and the Hessian \eqref{Hessian},
which is positive definite. Therefore, $\mathbb{J}_P$ is strictly convex, and any stationary point is a solution to the moment equations \eqref{momentconditions}, which must be a unique if it exists. Hence $F^P$ is injective.
\end{proof}

It remains to show that there exists a solution to the moment equations \eqref{momentconditions}.

\begin{lemma}\label{boundedlem}
Suppose the Toeplitz matrix $\Tb_n$ is positive definite; i.e., $\cb\in\mathfrak{C}_+$. Then, for any compact $K\subset \mathfrak{C}_+(N)$, the inverse image  $(F^P)^{-1}(K)$ is bounded.
\end{lemma}

\begin{proof}
Suppose $Q$ satisfies the moment equations $F^P(Q)=\cb$ for some $\cb\in\mathfrak{C}_+(N)$. Then
\begin{displaymath}
\langle C,Q\rangle =\sum_{k=-n}^n c_k\bar{q}_k = \int_{-\pi}^\pi P(e^{i\theta})d\nu =: \kappa,
\end{displaymath}
where $\kappa$ is a constant. Now, let $a(z)=a_0z^n+\dots +a_{n-1}z +a_n$ be the stable polynomial spectral factor of $Q(z)$, i.e., $a(z)a(z)^*=Q(z)$. Then 
\begin{displaymath}
\kappa=\langle C,Q\rangle =\mathbf{a}^* T_n\mathbf{a},
\end{displaymath}
where $\Tb_n$ is the Toeplitz matrix of  $\mathbf{c}$ and $\mathbf{a}=(a_0,a_1,\dots,a_n)\Tr$. If $\cb$ is restricted to the compact subset $K\in\mathfrak{C}_+$, the eigenvalues of $T_n$ are bounded away from zero. Hence $T_n\geq \varepsilon I$ for some $\epsilon >0$, and consequently 
\begin{displaymath}
\|\mathbf{a}\|^2 \leq \frac{1}{\varepsilon}\mathbf{a}^* T_n\mathbf{a}=\frac{\kappa}{\varepsilon}.
\end{displaymath}
Consequently, $\|\qb\|$, where $\qb=(q_0,q_1,\dots,q_n)$, is also bounded, and hence so is
$(F^P)^{-1}(K)$.
\end{proof}

\begin{lemma}\label{Fproper}
The moment map $F^P:\,\mathfrak{P}_+(N)\to \mathfrak{C}_+(N)$ is proper. 
\end{lemma}

\begin{proof}
Let $K$ be a compact subset of $\mathfrak{C}_+(N)$, and let $\cb^{(k)}$ be a sequence in $K$ converging to $\hat{\cb}\in K$. Since  $(F^P)^{-1}(K)$ is bounded (Lemma~\ref{boundedlem}), there is a convergent sequence $Q^{(k)}$ in the preimage of the sequence $\cb^{(k)}$ converging to some limit $\hat{Q}$. We want to show that $\hat{Q}\in (F^P)^{-1}(K)$. The only way this can fail is that $\hat{Q}$ belongs to  $\partial\mathfrak{P}_+(N)$, the boundary of   $\mathfrak{P}_+(N)$.  We observe that 
\begin{displaymath}
\langle C^{(k)},P\rangle =  \int_{-\pi}^\pi \frac{P^2}{Q^{(k)}}d\nu, 
\end{displaymath}
and consequently, since $P\in\mathfrak{P}_+(N)$,
\begin{displaymath}
\sum_{j=-N+1}^N  \frac{P(\zeta_j)^2}{\hat{Q}(\zeta_j)} =\langle \Hat{C},P\rangle,
\end{displaymath}
which requires that  $\hat{Q}(\zeta_j)\ne 0$ for all $j$. However, $\hat{Q}$ can only belong to $\partial\mathfrak{P}_+(N)$ if some $\hat{Q}(\zeta_j)$ equals zero. Hence $\hat{Q}\not\in\partial\mathfrak{P}_+(N)$, as required.
\end{proof}

This concludes the proof of Theorem~\ref{mainthm}. 

\subsection{Proof of  Theorem~\ref{homeomorphism}}

We have already proven above that $F^P$ is a homeomorphism. It remains to prove that $G^c$ is. For this we need two more lemmas.

\begin{lemma}
For each fixed $\cb\in\mathfrak{C}_+(N)$, the map $G^c:\, \mathfrak{Q}_+(N)\to \mathfrak{P}_+(N)$ is injective.
\end{lemma}

\begin{proof}
Suppose that $G^c(Q_1)=G^c(Q_2)=P$ for some $Q_1,Q_2\in\mathfrak{Q}_+(N)$. We want to show that $Q_1=Q_2$. To this end, since
\begin{displaymath}
\int_{-\pi}^\pi e^{ik\theta} \frac{(Q_1-Q_2)P}{Q_1Q_2}d\nu=0,\quad k=0,1,\dots,n,
\end{displaymath}
we have
\begin{displaymath}
\int_{-\pi}^\pi \frac{(Q_1-Q_2)^2P}{Q_1Q_2}d\nu=\sum_{k=-n}^n\left(\left[ q_k^{(1)}-q_k^{(2)}\right]\int_{-\pi}^\pi e^{ik\theta} \frac{(Q_1-Q_2)P}{Q_1Q_2}d\nu\right) =0,
\end{displaymath}
where
\begin{displaymath}
Q_\ell(z) =\sum_{k=-n}^n q_k^{(\ell)}z^k,\quad \ell=1,2,
\end{displaymath}
and, consequently, $Q_1(\zeta_j) =Q_2(\zeta_j)$ for all $j$, as claimed. 
\end{proof}

\begin{lemma}
For each fixed $\cb\in\mathfrak{C}_+(N)$, the map $G^c:\, \mathfrak{Q}_+(N)\to \mathfrak{P}_+(N)$ is proper.
\end{lemma}

\begin{proof}
The proof follows the same pattern as that of Lemma~\ref{Fproper}. Hence, let  $K$ be a compact subset of $\mathfrak{P}_+(N)$, and let $P^{(k)}$ be a sequence in $K$ converging to $\hat{P}\in K$. Since  $(G^c)^{-1}(K)\subset \mathfrak{Q}_+(N)$ is bounded (Lemma~\ref{boundedlem}), there is a convergent sequence $Q^{(k)}$ in the preimage of the sequence $P^{(k)}$ converging to some limit $\hat{Q}$. In order to ensure that $\hat{Q}\in (G^c)^{-1}(K)$, we must demonstrate that $\hat{Q}\not\in\partial\mathfrak{P}_+(N)$. To this end, note that 
\begin{displaymath}
\langle C,P^{(k)}\rangle =  \int_{-\pi}^\pi \frac{(P^{(k)})^2}{Q^{(k)}}d\nu, 
\end{displaymath}
and consequently, since $c\in\mathfrak{C}_+(N)$ and $\hat{P}\in K\subset\mathfrak{P}_+(N)$,
\begin{displaymath}
\sum_{j=-N+1}^N  \frac{\hat{P}(\zeta_j)^2}{\hat{Q}(\zeta_j)} =\langle C,\hat{P}\rangle >0,
\end{displaymath}
Since  $\hat{P}(\zeta_j)>0$ for all $j$, this requires that  $\hat{Q}(\zeta_j)>0$ for all $j$. Hence $\hat{Q}\not\in\partial\mathfrak{P}_+(N)$, as required.
\end{proof}

The map $G^c$ is a continuous map between connected spaces of  the same dimension $n+1$. Noting that \eqref{Hessian} is positive definite, the continuity follows from  the inverse function theorem applied to the equation $F^P(Q)=\cb$. Then, since $G^c$ is injective and proper, it follows from Theorem 2.6 in \cite{BLvariational}, that it is a homeomorphism.

\subsection{Proof of  Theorem~\ref{dualboundarythm}}

We follow the lines of the proof of the Main Lemma in \cite[p. 10]{BLmoments}.
Let $(P_\ell)$ be a sequence in $\mathfrak{P}_+(N)$ converging to $P\in\overline{\mathfrak{P}_+(N)}\setminus \{0\}$. Then there is a positive constant $K$ such that $P_\ell(\zeta_j)\leq K$ for  $\ell=1,2,3,\dots$ and $j=-N+1,\dots,N$. For each $\ell$, let $Q_\ell$ be the unique minimizer of
\begin{displaymath}
\mathbb{J}_{P_\ell}(Q)= \langle C,Q\rangle -\int_{-\pi}^\pi  P_\ell(e^{i\theta})\log Q(e^{i\theta})d\nu
\end{displaymath}
as prescribed by Theorem~\ref{optthm}. Then
\begin{equation}
\label{momentsell}
\int_{-\pi}^\pi e^{ik\theta}\frac{P_\ell(e^{i\theta})}{Q_\ell(e^{i\theta})}d\nu =c_k, \quad k=1,2,\dots,n.
\end{equation}
Now suppose that the sequence $Q_\ell$ is unbounded. Then there is a subsequence, which we shall also denote $Q_\ell$, for which $\| Q_\ell\|_\infty>1$ and $\| Q_\ell\|_\infty\to\infty$. For each such $Q$, there is an $\varepsilon>0$ such that 
\begin{equation}
\label{Jbound}
\mathbb{J}_{P_\ell}(Q)\geq \varepsilon\|Q\|_\infty-K\log\|Q\|_\infty.
\end{equation}
To see this, first note 
that, since $T_n>0$, $\langle C,Q/\|Q\|_\infty\rangle$ has a minimum $\varepsilon>0$ on the compact set $\{ Q\in \overline{\mathfrak{P}_+(N)}\mid \|Q\|_\infty=1\}$, we have $\langle C,Q\rangle\geq \varepsilon\|Q\|_\infty$. Then
\begin{displaymath}
\mathbb{J}_{P_\ell}(Q)\geq \varepsilon\|Q\|_\infty-\int_{-\pi}^\pi P_\ell\log\left(\frac{Q}{\|Q\|_\infty}\right)d\nu-K\log\|Q\|_\infty,
\end{displaymath}
where the second term is nonnegative and can be deleted. 

Next, let $\tilde{Q}\in\mathfrak{P}_+(N)$ be arbitrary. Then, by optimality, $\mathbb{J}_{P_\ell}(\tilde{Q})\geq\mathbb{J}_{P_\ell}(Q_\ell)$. Since $\mathbb{J}_{P_\ell}(\tilde{Q})\to\mathbb{J}_P(\tilde{Q})$ as $\ell\to\infty$, there is a positive constant $L$ such that 
\begin{equation*}
L\geq \mathbb{J}_{P_\ell}(\tilde{Q})\geq\mathbb{J}_{P_\ell}(Q_\ell), \quad \ell=1,2,3,\dots,
\end{equation*}
which together with \eqref{Jbound} yields
\begin{equation}
\label{ }
L\geq\varepsilon\|Q_\ell\|_\infty-K\log\|Q_\ell\|_\infty, \quad \ell=1,2,3,\dots.
\end{equation}
Then, comparing linear and logarithmic growth, we see that the sequence $(Q_\ell)$ is bounded, contrary to hypothesis.  Consequently, there is a convergent subsequence (for convenience also indexed by  $\ell$) such that $Q_\ell\to\hat{Q}$, and,  since  $P_\ell\to P\ne 0$, \eqref{momentsell} implies that $\hat{Q}\ne 0$. Hence, setting $\Phi_\ell:=P_\ell/Q_\ell$ and $\hat{\Phi}:= P/\hat{Q}$, 
$\Phi_\ell\to\hat{\Phi}$, 
and hence, taking limits in \eqref{momentsell}, we obtain
\begin{displaymath}
\int_{-\pi}^\pi e^{ik\theta}\hat{\Phi}(e^{i\theta})d\nu =c_k, \quad k=1,2,\dots,n,
\end{displaymath}
showing that $\hat{Q}$ is the required minimizer satisfying the moment conditions.

\subsection{Proof of Theorem~\ref{PQthm}}

We begin by showing that the sublevel set $\mathbb{J}^{-1}(\infty,r]$ is compact for each $r\in\mathbb{R}$. The sublevel set consists of those $(P,Q)\in\overline{\mathfrak{P}_+^\mathbf{o}}(N)\times\overline{\mathfrak{P}_+(N)}$ for which 
\begin{displaymath}
r\geq \mathbb{J}_1(P,Q) +\mathbb{J}_2(P),
\end{displaymath}
where
\begin{align*}
\label{}
 &\mathbb{J}_1(P,Q)   =  \langle C,Q\rangle-\int_{-\pi}^\pi  P(e^{i\theta})\log\det Q(e^{i\theta})d\nu \\
& \mathbb{J}_2(P) =  -\langle M,P\rangle +\int_{-\pi}^\pi  P(e^{i\theta})\log\det P(e^{i\theta})d\nu
\end{align*}
Since $\mathfrak{P}_+^\mathbf{o}(N)$ is a bounded set that is bounded away from zero, there is a positive constant $K$ such that $\|P\|_\infty \leq K$ and a $\rho\in\mathbb{R}$ such that $\mathbb{J}_2(P)\geq \rho$ for all $P\in \mathfrak{P}_+^\mathbf{o}(N)$. Hence, in view of the estimates leading to \eqref{Jbound}, 
\begin{displaymath}
r-\rho\geq \mathbb{J}_1(P,Q) \geq \varepsilon \|Q\|_\infty -K\log\|Q\|_\infty,
\end{displaymath}
and therefore, comparing linear and logarithmic growth, it follows that the sublevel set $\mathbb{J}^{-1}(\infty,r]$ is bounded. Since it is also closed, it is compact, as claimed. 

Since $\mathbb{J}$ thus has compact sublevel sets, there is a minimizer $(\hat{P},\hat{Q})$. Then clearly $\hat{Q}$ is a minimizer of $\mathbb{J}_{\hat{P}}$, and hence, by Theorem~\ref{dualboundarythm}, $\hat{\Phi}:=\hat{P}/\hat{Q}$ satisfies the moment conditions \eqref{momentconditions}. If $\hat{P}\in\mathfrak{P}_+^\mathbf{o}(N)$, then the minimizer must satisfy the stationarity condition $\partial\mathbb{J}/\partial \bar{p}_k =0$, $k=1,2,\dots,N$, and hence, by \eqref{gradP}, $\hat{\Phi}$ also satisfies the  logarithmic moment conditions \eqref{cepstrum}. Since 
\begin{displaymath}
\mathbb{I}(\hat{P},\hat{Q})=L(\hat{P},\hat{Q})\geq L(P,Q) \quad\text{for all $(P,Q)$},
\end{displaymath}
and $L(P,Q)=\mathbb{I}(P,Q)$ for all $(P,Q)$ satisfying the moment conditions \eqref{momentconditions}, $(\hat{P},\hat{Q})$ solves the primal problem. By Theorem~\ref{mainthm}, $\hat{Q}\in\mathbb{P}_+(N)$.  

It remains to prove that the optimal solution is unique if $(\hat{P},\hat{Q})\in\mathfrak{P}_+^\mathbf{o}(N)\times\mathfrak{P}_+(N)$. To this end, we form the directional derivative
\begin{displaymath}
\delta\mathbb{J}(\hat{P},\hat{Q};\delta P,\delta Q)=\langle C-\hat{P}\hat{Q}^{-1},\delta Q\rangle + \langle \log(\hat{P}\hat{Q}^{-1})-M, \delta P\rangle
\end{displaymath}
and the second directional derivative
\begin{displaymath}
\delta^2\mathbb{J}(\hat{P},\hat{Q};\delta P,\delta Q)=\langle (\delta P-\hat{P}\hat{Q}^{-1}\delta Q), \hat{P}^{-1}(\delta P-\hat{P}\hat{Q}^{-1}\delta Q)\rangle \geq 0
\end{displaymath}
with equality if and only if 
\begin{displaymath}
\delta P-\hat{P}\hat{Q}^{-1}\delta Q=0.
\end{displaymath}
Then, however, 
\begin{equation*}
\int_{-\pi}^\pi \hat{P}\hat{Q}^{-1}\delta Q\,d\nu= \int_{-\pi}^\pi õ\delta P\,d\nu=0,
\end{equation*}
since  the pseudo-polynomial $\delta P$ has no constant term, as $P(0)=1$. Therefore, choosing $\delta Q=1$, it follows from Theorem~\ref{mainthm} that 
\begin{displaymath}
c_0=\int_{-\pi}^\pi \hat{P}\hat{Q}^{-1}\,d\nu=0,
\end{displaymath}
which is a contradiction. Consequently,
\begin{displaymath}
\delta^2\mathbb{J}(\hat{P},\hat{Q};\delta P,\delta Q)>0
\end{displaymath}
for all $\delta P, \delta Q$; i.e., the Hessian of $\mathbb{J}$  is positive definite, and hence $\mathbb{J}$ is strictly convex. Therefore uniqueness follows.


\begin{thebibliography}{99}

\bibitem{BLGM1} C. I. Byrnes,  A. Lindquist, S.V. Gusev,  and A. 
V. Matveev, A complete parameterization of all positive rational extensions of a covariance sequence, {\em IEEE Trans. Aut. Contr.} {\bf AC-40} (1995) 1841-1857.

\bibitem{Byrnes-L-97}
C. I. Byrnes and A. Lindquist, On the partial stochastic realization problem, {\em IEEE Transactions on Automatic Control} {\bf AC-42} (1997), 1049--1069. 

\bibitem{BGuL} C. I. Byrnes,  S. V. Gusev,  and A. Lindquist,
`A convex optimization approach to the rational covariance
extension problem,
{\em SIAM J. Control and Opt.} {\bf 37} (1999), 211-229.

\bibitem{SIGEST} C. I. Byrnes,  S.V. Gusev,  and A. Lindquist,
From finite covariance windows to modeling filters: A convex optimization
approach, {\em SIAM Review} {\bf 43} (2001) 645--675.

\bibitem{BEL1} C. I. Byrnes,  P. Enqvist,  and A. Lindquist,
{\em Cepstral coefficients, covariance lags and pole-zero models for finite data
strings}, IEEE Trans.\ on Signal Processing {\bf SP-50} (2001), 677--693.

\bibitem{BEL2} C. I. Byrnes,  P. Enqvist,  and A. Lindquist, {\em Identifiability and 
well-posedness of shaping-filter parameterizations: A global analysis approach},
SIAM J. Control and Optimization, {\bf 41} (2002), 23--59.

\bibitem{BLmoments}
C. I. Byrnes and A. Lindquist,
The generalized moment problem with complexity constraint, {\em Integral Equations and Operator Theory} {\bf 56} (2006) 163--180.

\bibitem{BLvariational}
C. I. Byrnes and A. Lindquist, Interior point solutions of variational problems and global inverse function theorems, {\em International Journal of Robust and Nonlinear Control} {\bf 17} (2007), 463--481.

\bibitem{BLimportantmoments} C. I. Byrnes and A. Lindquist, Important moments in systems and control, {\em  SIAM J. Control and Optimization} {\bf 47}(5) (2008),  2458--2469.

\bibitem{BLkrein}
C. I. Byrnes and A. Lindquist, The moment problem for rational measures: convexity in the spirit of Krein,  in {\em Modern Analysis and Application: Mark Krein Centenary Conference},  Vol. I:  Operator Theory and Related Topics, Book Series: Operator Theory Advances and Applications Volume 190,  Birkh{\"a}user, 2009, pp. 157 -- 169.

\bibitem{Carli-FPP}
F. P. Carli and A. Ferrante and M. Pavon and G. Picci, A Maximum Entropy Solution of the Covariance Extension Problem for Reciprocal Processes, {\em IEEE Trans. Automatic Control} {\bf AC-56} (2011), 1999-2012.

\bibitem{CarliGeorgiou}
F. Carli, T.T. Georgiou, On the Covariance Completion Problem under a Circulant Structure, {\em IEEE Transactions on Automatic Control} {\bf 56}(4) (2011), pp. 918 -922.

\bibitem{Chiuso-F-P-05}
A. Chiuso and A. Ferrante and G. Picci, Reciprocal realization and modeling of textured images, {\em Proceedings of the 44rd IEEE Conference on Decision and Control}, 2005. 

\bibitem{Davis-79}
P. Davis, {\em Circulant Matrices}, John Wiley \& Sons, 1979.

\bibitem{DS}
N. Dunford and J.T. Schwartz, {\em Linear Operators, Part I: General Theory}, John Wiley \& Sons, New York, 1958.

\bibitem{Dym-G-81}
H. Dym and I. Gohberg, Extension of band matrices with band inverses, {\em Linear Algebra and Applications} {\bf 36} (1981), 1-24.

\bibitem{PEthesis}
P. Enqvist, {\em Spectral estimation by Geometric, Topological and Optimization Methods}, 
PhD thesis, Optimization and Systems Theory, KTH, Stockholm,
Sweden, 2001. 

\bibitem{PE}
P. Enqvist, A convex optimization approach to ARMA(n,m) model design from covariance and cepstrum data,
{\em SIAM Journal on Control and Optimization}, {\bf 43(3)}: 1011-1036, 2004.

\bibitem{Ferrante-P-R-07}
A. Ferrante and M. Pavon and F. Ramponi, Further results on the {B}yrnes-{G}eorgiou-{L}indquist generalized moment problem, {\em Modeling, Estimation and Control: Festschrift in honor of Giorgio Picci on the occasion of his sixty-fifth Birthday}, Springer-Verlag, 2007.

\bibitem{Frezza-90}
R. Frezza, {\em Models of Higher-order and Mixed-order Gaussian Reciprocal Processes with Application to the Smoothing Problem}, PhD thesis, Applied Mathematics Program, U.C.Davis, 1990.

\bibitem{Gthesis} 
T.T. Georgiou, {\em Partial Realization of Covariance Sequences}, Ph.D. thesis, CMST, University of Florida, Gainesville 1983.

\bibitem{Georgiou1} T.T. Georgiou, Realization of power spectra from partial covariances, {\em IEEE Trans. on Acoustics, Speech and Signal Processing} {\bf ASSP-35} (1987) 438-449.

\bibitem{Georgiou3} T.T. Georgiou, Solution of the general moment problem via a one-parameter imbedding,  {\em IEEE Trans. Aut. Contr.} {\bf AC-50} (2005) 811-826.

\bibitem{Georgiou-L-03}
T. T. Georgiou and A. Lindquist, Kullback-Leibler approximation of spectral density functions, {\em IEEE Trans. Information Theory} {\bf 49} (2003), 2910--2917. 

\bibitem{Gohberg-G-K-94}
I. Gohberg,  S. Goldberg and M. Kaashoek, {\em Classes of Linear Operators}, Vol. II, 
Birkh\"auser, Boston, 1994.

\bibitem{Gray-02}
R. M. Grey, {\em Toeplitz and Circulant Matrices: A Review}, Stanford University (http:{//}ee.stanford.edu{/~}gray{/}toeplitz.pdf), 2002. 

\bibitem{Hormander}
L. H\"ormander, {\em An Introduction to Complex Analysis in Several Complex Variables}, Noth Holland, 1966. 

\bibitem{Jamison-70}
B. Jamison, Reciprocal Processes: The stationary Gaussian case, {\em Ann. Math. Stat.} {\bf 41}, 1624-1630.

\bibitem{Jamison-74}
B. Jamison, Reciprocal Processes, {\em Zeitschrift. Wahrsch. Verw. Gebiete} {\bf 30} (1974), 65-86. 

\bibitem{Kalman} R. E. Kalman, Realization of Covariance Sequences, Proc. Toeplitz Memorial Conference, Tel Aviv, Israel, 1981.

\bibitem{KreinNudelman}
M.G. Krein and A.A. Nudelman,
{\em The Markov Moment Problem and Extremal Problems},  American Mathematical
Society, Providence, Rhode Island, 1977.

\bibitem{Krener-86}
A.J. Krener, Reciprocal Processes and the stochastic realization problem for acausal systems, {\em Modeling Identification and Robust Control}, C.I. Byrnes and A. Lindquist, eds., Noth-holland, 1986, 197--211.

\bibitem{Krener-86b}
A.J. Krener, Realization of Reciprocal Processes, Proc IIASA Conf. on Modeling and Adaptive Control, Springe-Verlag, 1986.

\bibitem{Levy-92}
B. C. Levy, Regular and Reciprocal Multivariate Stationary {G}aussian Reciprocal Processes over {{\bf Z}} are Necessarily {M}arkov, {\em J. Math. Systems, Estimation and Control} {\bf 2} (1992), 133--154.

\bibitem{Levy-F-02}
B. C. Levy and A. Ferrante, Characterization of stationary discrete-time {G}aussian Reciprocal Processes over a finite interval, {\em SIAM J. Matrix Anal. Appl.} {\bf 24} (2002), 334-355. 

\bibitem{Levy-F-K-90}
B. C. Levy and R. Frezza and A.J. Krener, Modeling and Estimation of discrete-time {G}aussian Reciprocal Processes, {\em IEEE Trans. Automatic Control} {\bf AC-35} (1990), 1013-1023.

\bibitem{Masani-60}
P. Masani, The prediction theory of multivariate stochastic proceses, III, {\em Acta Mathematica} {\bf 104} (1960), 141-162.

\bibitem{Musicus}
B. R. Musicus and A. M. Kabel, Maximum entropy pole-zero estimation, Technical Report 510, MIT Research Laboratory of Electronics, August 1985.

\bibitem{Pavon-F-12}
M. Pavon and . Ferrante, {\em On the Geometry of Maximum Entropy
Problems}, provisionally accepted for publication in {\em SIAM REVIEW},
available in {\tt http://arxiv.org/abs/1112.5529}, 2012.

\bibitem{Picci-C-08}
G. Picci and F. Carli, Modelling and simulation of images by reciprocal processes, {\em Proc. Tenth International Conference on Computer Modeling and Simulation UKSIM 2008}, 513--518.

\bibitem{rodman2002abstract}
L. Rodman, I.M.  Spitkovski{\u\i} and H.J. Woerdeman, {\em Abstract Band Method Via Factorization Positive and Band Extensions of Multivariable Almost Periodic Matrix Functions, and Spectral Estimation}, Memoirs of the American Mathematical Society, 2002. 


\end{thebibliography}
\end{document}